\documentclass[letterpaper, 11pt]{amsart}

\usepackage{amssymb, amsthm}
\usepackage{mathtools} 
\usepackage{amsaddr}
\usepackage[numbers]{natbib}
\usepackage{comment}
\usepackage{enumerate}
\usepackage{hyperref} 

\theoremstyle{plain} 
	\newtheorem{thm}{Theorem}[section]
	\newtheorem{lemma}[thm]{Lemma} 
	\newtheorem{prop}[thm]{Proposition}
	\newtheorem{cor}[thm]{Corollary}

\theoremstyle{definition}
	\newtheorem{definition}[thm]{Definition} 
	\newtheorem{notation}[thm]{Notation} 
	\newtheorem{example}[thm]{Example} 

\theoremstyle{remark}
	\newtheorem{remark}[thm]{Remark}   

\DeclareMathOperator{\Spec}{Spec} 
\DeclareMathOperator{\Max}{Max}
\DeclareMathOperator{\Proj}{Proj}
\DeclareMathOperator{\Zar}{Zar}

\newcommand{\Abhy}{Abhyankar}

\newcommand{\complexC}{\mathbb{C}}
\newcommand{\naturalN}{\mathbb{N}}
\newcommand{\inverse}{^{-1}} 
\newcommand{\set}[2]{\{#1\mid #2\}}
\newcommand{\ideal}{\trianglelefteq} 
\newcommand{\gendby}[1]{\langle #1 \rangle} 
\newcommand{\eps}{\epsilon}
\newcommand{\fraka}[1][a]{\mathfrak{#1}}

\newcommand{\localDomains}[1][D]{\mathcal{L}(K/#1)}
\newcommand{\curlyS}{\mathcal{S}}
\newcommand{\curlySOf}[1]{\curlyS \left( {#1} \right)}
\newcommand{\projModel}{\mathcal{W}}
\newcommand{\modelX}{\mathcal{X}}
\newcommand{\modelY}{\mathcal{Y}}
\newcommand{\abyideal}{{\mathcal{I}}}
\newcommand{\sheafO}{{\mathcal{O}}}
\newcommand{\sheafideal}[1][F]{{{\mathcal{#1}}}} 

\newcommand{\dstar}{D^{\ast}}
\newcommand{\dstarof}[1][\modelX]{D_{#1}^{\ast}}
\DeclareMathOperator{\AffineSet}{Aff}
\newcommand{\affinesInX}{\AffineSet (\modelX)} 

\newcommand{\dominatedby}{\preccurlyeq}
\newcommand{\centermap}{\delta}
\newcommand{\centerFrom}[2]{{\centermap}^{#1}_{#2}}

\title{Viewing Quasi-Coherent Sheaves of Ideals as Ideals of a Ring}
\author{Ay\c{c}in Iplik\c{c}i Arodirik}
\email{iplikciarodirik.1@osu.edu}
\address{Department of Mathematics, The Ohio State University, \\Columbus, OH 43210}

\date{\today} 
\begin{document}
	\begin{abstract}
		This paper presents a technique for viewing quasi-coherent sheaves of ideals of a given blowup as regular ideals of a ring. In the paper, we first describe (Zariski) models as integral schemes that are separated and of finite type over an integral domain $D$. We then construct a ring $D^*$ for a given projective model (e.g. blowup of $D$ over a finitely generated ideal) by intersecting Nagata function rings. The spectrum of $D^*$ contains the projective model, but similar to the Proj-construction, it includes additional prime ideals. We characterize the relevant ideals and construct a faithfully flat morphism of schemes from the spectrum of $D^*$ to the model. Finally, using Abhyankar's definition of ideals on models, we identify the relevant ideals of $D^*$ with the quasi-coherent sheaves of ideals of the corresponding projective model.
	\end{abstract}

\maketitle

\section*{Introduction}
A projective scheme can be constructed as a \( \Proj \) of a graded ring. When the graded ring is taken as a polynomial ring \( D [t] \) of an integral domain \( D \), where \( t \) is indeterminate, \( \Proj D[t] \) and \( \Spec D \) are homeomorphic to each other. In other words, the homogeneous prime ideals of \( D[t] \) that are not irrelevant display the spectrum of \( D \). Instead of a polynomial extension, if we consider the Nagata extension \( D(t) \), then the maximal spectrum of \( D(t) \) presents the maximal spectrum of \( D \). But, in general we do not have a bijective correspondence between the spectra of \(  D(t) \) and \(  D \). We want to characterize “irrelevant ideals” of \( D(t) \), so that the “relevant prime spectrum” of \( D(t) \) will reflect \( \Spec D \).  Working with the Nagata extensions rather than the polynomial extensions allows us to encapsulate the local affine pieces of a projective scheme (along with some irrelevant ideals) into a single ring. More precisely, (Theorem \ref{thm. flat morphism from Spec Dstar to projective X}) for a given blow-up \( \modelX \) of an integral domain \( D \) along a finitely generated ideal, we construct a ring \( \dstarof[\modelX] \) and a morphism of schemes from \( \Spec \dstarof[\modelX] \) to \( \modelX \) that is faithfully flat. As an application (Corollary~\ref{thm. relevant ideals corresponds q-coh sheaf of ideals}) we show that the quasi-coherent ideals on \( \modelX \) are in bijective correspondence with the relevant ideals of \( \dstarof[\modelX] \). 

Instead of working with schemes, we will use the terminology of (Zariski) models as in \cite{ZariskiSamuel2}. We regard the spectrum of a ring as the set of localizations at its prime ideals rather than as the set of prime ideals. In this sense, a model over an integral domain \( D \) is a union of finite spectra of finitely generated \( D \)-algebras such that no two localizations are dominated by the same valuation domain. Section~\ref{Sec. Prelim Model} provides related terminology on the models and Subsection~\ref{Subsec. top on models} recalls some properties about the topological structure of models. In essence, a model over \( D \) is an irreducible, reduced, and separated \( D \)-scheme of finite type. Finally, Subsection~\ref{Subsec. Abhy ideals} introduces \Abhy{}'s definition of ideals on models appearing in \cite{Abh_resolutionBook}, which corresponds to the quasi-coherent sheaves of ideals. 

The spectrum of a finitely generated \( D \)-algebra is called an affine model over \( D \), and it is apparent that an affine model is also an affine scheme with some extra properties. Likewise, projective models are also projective schemes. These observations can be found in prior literature on the subject (e.g. \cite{heinzerolb2020noetIntDimtwo}, \cite{heinzer2016blowingCompleteIdeals}, \cite{olberding2017principalIdealThmforVal}).
Let \( \Zar(K/A) \) denote the set of valuations of \( K \) containing \( A \). \citet{koji1990pruferDomAffineSch} shows \( \Zar(K/A) \) is a locally ringed space. It is known that \( \Zar(K/A) \) is the projective limit of the projective models (see \citet{ZariskiSamuel2}). \citet{koji1993ringedSpVal} expresses and proves this within scheme terminology when the ground ring \( A \) is a Noetherian ring. Moreover, \citet{olberding2015affineSchemes} characterizes subspaces of \( \Zar(K/A) \)  that correspond to affine schemes. Section~\ref{Sec. models as schemes} connects the terminologies and translates the properties of models into their scheme counterparts. The proofs in further sections will mainly use the model terminology but the statements might be expressed within the scheme terminology.

Section~\ref{Sec. Dstar of models} presents the aforementioned ring construction. For a model \( \modelX \),
we set \( \dstarof[\modelX] \) as the intersection of the Nagata extension of local rings from \( \modelX \), that is 
\begin{equation*}
	\dstarof[\modelX]  = \bigcap_{R \in \modelX} R(t)
\end{equation*}
where \( R(t) \) is the Nagata extension of \( R \).
The Nagata extension of a ring \( R \) is the localization of the polynomial ring \( R[t] \) at the set of all primitive polynomials; that is, a polynomial whose coefficients generate the unit ideal \( R \). Subsection~\ref{Subsec. construction Dstar of models} explores this construction when \( \modelX \) corresponds to a projective model, while Subsection~\ref{Subsec. rel Ideals of Dstar} defines a generalization of the concept of homogeneous ideals, from polynomial rings to Nagata extensions and then to \( \dstarof[\modelX] \). We call an ideal of \( D(t) \) relevant if it is an extension of an ideal from \( D \). As \( t \) is a unit in \( D(t) \), the relevant prime spectrum is exactly \( \Proj D[t] \). In Section~\ref{Sec. Abh vs rel ideals}, we demonstrate that the quasi-coherent sheaves of ideals on a projective model can be interpreted as the relevant ideals of the constructed ring.

By a ring, we mean a commutative ring with identity. Throughout this paper, \( D \) refers to an integral domain with \( K \) a field containing \( D \) as a subring.

\section{Preliminaries on Models} \label{Sec. Prelim Model}

We recall the necessary definitions and facts regarding the model terminology. For more details, the reader is referred to \citet[Ch. VI \S 17]{ZariskiSamuel2} and \citet[Ch. I \S 1]{Abh_resolutionBook}. 
The set of all local domains that contain \( D \) and have fraction field \( K \) is denoted by \( \localDomains \).  The following notations will be frequently used. 

\begin{notation} 
	For a ring \( A \), the set of localizations of \( A \) at its prime ideals is denoted by \( \curlyS (D) \); that is,
	\begin{equation*}
		\curlyS (A) = \set{A_P}{P\ideal A \text{ is prime} }.
	\end{equation*}
	For a subset \( H \) of \( K \) containing a nonzero element, we write
	\begin{equation*}
		\projModel(D;H) = \bigcup_{0 \neq x\in H} \curlySOf{D[Hx\inverse]}
	\end{equation*}
	where \( D[H x\inverse] \) is the subring of \( K \) generated (as a ring) by the elements of \( D \) and the elements of the form \( y x\inverse \) for \( y\in H \). We also write \( \projModel(D;x_1,\dotsc ,x_n) \) when \( H=\{x_1,\dotsc ,x_n\} \) is a finite set.
\end{notation}
The notation \( \curlySOf{A} \) is referred to as \( \mathfrak{V}(A) \) in~\cite{Abh_resolutionBook} and as \( V(A) \) in \cite{ZariskiSamuel2}, with the latter omitting the field of fractions of \( A \). In the following section, we consider \( \localDomains \) as a topological space, and the induced topology on \( \curlySOf{A} \) will correspond to the usual Zariski spectrum topology if we regard \( \curlySOf{A} \) as \( \Spec A \) by identifying a prime ideal \( P \) with the local ring \( A_P \).  In general, the inclusion reversing nature of the identification prevents \( \curlySOf{\cdot} \) from being a closed set, whereas, \( V(\cdot) \) is typically associated with closed sets in the spectrum topology. We change the notation to avoid any potential confusion and to highlight its identification with the spectrum. 

Although both \cite{ZariskiSamuel2} and \cite{Abh_resolutionBook} define a local ring as a Noetherian (quasi-)local ring, we follow the convention that a local ring is a ring with a unique maximal ideal, not necessarily Noetherian. We adjust the definitions accordingly. Around the terminology of models, \cite{Abh_resolutionBook} is mainly consistent with \cite{ZariskiSamuel2} with a few slight technical variations (see Remark~\ref{rmk. different versions for model def}). The adapted definitions primarily follow~\cite{Abh_resolutionBook}.

\begin{definition} \phantom{pushing the items to the next line}
	\begin{itemize}
		\item Given two local domains \( (R, M_R) \) and \( (S, M_S) \), we say 
		\( S \) dominates \( R \) if \( R \) is a subring of \( S \)	satisfying \( M_R = M_S\cap R \). 
		\item A premodel of \( K \) over \( D \) is a non-empty subset of \( \localDomains \).
		\item A premodel \( \modelX \) is called irredundant if no two distinct elements of \( \modelX \) are dominated by the same valuation ring; equivalently if no distinct two elements of \( \modelX \)  can be dominated by the same local domain.  
		\item A premodel \( \modelX \) is called complete if for any valuation ring \( V \) of \( K \) that contains \( D \), there exists an element of \( \modelX \) that is dominated by~\( V \).
		\item A subring \( A \) of \( K \) that contains \( D \) is called an affine ring over \( D \) if it is finitely generated as a \( D \)-algebra. In other words, there are \( x_1, \dotsc x_n \in K \) such that \( A=D[x_1, \dotsc x_n] \) i.e. \( A \) is the smallest subring of \( K \) that contains \( D \cup \{ x_1, \dotsc x_n \} \).  
		\item A model of \( K \) over \( D \) is an irredundant premodel \( \modelX \) of the form \( \bigcup_{\eps \in \Lambda } \curlySOf{D_{\eps}} \) where  \( \big\{ D_{\eps} \big\}_{\eps \in \Lambda }  \) is a finite family of affine rings over \( D \). 
		\item A projective model \( \modelX \) of \( K \) over \( D \) is a premodel of \( K \) over \( D \) of the form \( \projModel(D;H) \) for a finite subset \( H\subseteq K \) that contains a nonzero element. In such a case, we say \( \modelX \) is determined by \( H \). 
	\end{itemize}
\end{definition}

\begin{example}
	When \( K \) is the fraction field of \( D \), the premodel \(  \curlySOf{D} \) is a complete model of \( K \) over \( D \). 
	More generally, if \( A \) is an affine ring over \( D \) and \( K \) is the field of fractions of \( A \), then \( \curlySOf{A} \) is a model of \( K \) and called an affine model over \( D \).
\end{example}
Not all finite unions of spectra form a model. The following provides a non-example of models. 
\begin{example} \label{ex. non-irredundant premodel}
	Let \(  D = \complexC [x,y] \) and \( K = \complexC (x,y) \). Then \( \curlySOf{D}\cup \curlySOf{D[x/y]}  \) is a premodel of \( K \) over \( D \), but it is not a model as it is not irredundant. 
\end{example}

\begin{remark} \label{rmk. different versions for model def}
	We often omit the field \( K \) and refer to a model of \( K \) over \( D \) simply as a model over \( D \). This will not cause confusion since for a given model of \( K \) over \( D \), the field \( K \) is the unique field contained in the model. Note that this use is adopted from \cite{Abh_resolutionBook}. In \cite{ZariskiSamuel2}, a model \( \modelX \) of a field \( K' \) over \( D \) is defined in such a way that it does not contain \( K' \). Even with this version, all \( R\in \modelX \) have the same fraction field due to the irrendundancy condition and hence adding the common fraction field \( K \) to \( \modelX \) yields the version we adopted. 
\end{remark}

Contrary to the ambient field \( K \), the chosen ground ring \( D \) might affect the properties of a model of \( K \) over \( D \). 
\begin{example}
	For \(  D = \complexC [x,y] \), both \(  \curlySOf{D} \) and \(  \curlySOf{D[x/y]}  \) are affine models over \( D \). However, only the first one is complete. No element of the latter one is dominated by the valuation domain \( D[y/x]_{\gendby{x, y/x}} \). 
\end{example}

As the naming suggests, projective models are indeed models. We present some facts on projective models that will be used tacitly in the upcoming sections.
\begin{prop} \cite[1.7.3]{Abh_resolutionBook} 
	Let \( H \) be a subset of \( K \) that contains a nonzero element. If \( K \) is the field of fractions of \( D[H x\inverse] \) for some (hence for all) nonzero \( x\in H \), then \( \projModel(D;H) \) is an irredundant premodel of \( K \) over \( D \). If, additionally, \( H \) is finite, then \( \projModel(D;H) \) is a complete model.
\end{prop}

\begin{prop} \cite[1.9.3 \& 1.9.5]{Abh_resolutionBook}
	If \( H \) is a \( D \)-submodule of \( K \), then \( \projModel(D;H) = \projModel(D;\{x_\eps\}_{\eps \in \Lambda }) \) for any set of generators \( \{x_\eps\}_{\eps \in \Lambda } \) of \( H \). Moreover, when \( H \) is an ideal of \( D \), then
	\begin{equation*}
		\set{R\in \curlySOf{D}}{HR=R}=\set{R\in \projModel(D;H)}{HR=R}
	\end{equation*}
	holds.
\end{prop}

\begin{remark}
	Consider a projective model \( \modelX = \projModel(D;H) \) determined by a finite subset \( H \) of \( K \). Suppose that \( K \) is the fraction field of \( D \). Then, by clearing the denominators, we may assume that \( H \) is a finitely generated ideal of \( D \). More precisely, we first choose a nonzero \( x \in K \) such that  \( x H \subseteq D \). Such an element exists as \( H \) is a finite set; for instance \( x \) can be taken as the product of denominators of the elements in \( H \). Now, taking \( I \) as the ideal of \( D \) generated by the finite set \( x H \) yields \( \modelX = \projModel(D;I) \).
\end{remark}

\subsection{Zariski Topology on Models} \label{Subsec. top on models}
A premodel \( \modelX   \) has a topology where the sets of the form \( \set{R \in \modelX}{R \supseteq H} \) for a finite subset \( H \) of \( K \) form a basis. This topology is called the Zariski topology. In fact, the Zariski topology on \( \modelX \) is induced from the Zariski topology of \( \localDomains \). When \( \modelX = \curlySOf{D} \), this topology corresponds to the (spectrum) Zariski topology via the identification \( \curlySOf{D} \ni R=D_P \cong P \in \Spec(D) \). In this case, for nonzero elements \( f\in D \), the open sets \( \curlySOf{D[1/f]} \) correspond to the basic open sets in the (spectrum) Zariski topology.  

The affine rings \( A \) over \( D \) which satisfy \( \curlySOf{A} \subseteq \modelX \) have an important role in the structure of a model \( \modelX \) over \( D \). By definition such subsets are open in \( \modelX \) and moreover, they form a basis for open sets. In Section~\ref{Sec. models as schemes}, we observe that when a model \( \modelX \) is considered to be a scheme, they correspond to affine open subschemes. We introduce a notation for the collection of such rings. 

\begin{notation}
	For a given \( \modelX \subseteq \localDomains \), let \( \affinesInX \) denotes the set of affine rings over \( D \) such that \( \curlySOf{A} \subseteq \modelX \).
\end{notation}

\begin{definition} 
	We say a model \( \modelX \) over \( D \) is defined by the family \( \{D_{\eps}\}_{\eps \in \Lambda } \) if \( \modelX = \bigcup_{\eps \in \Lambda } \curlySOf{D_{\eps}} \) and \( D_{\eps} \) is in \( \affinesInX \) for all \( \eps \). 
\end{definition}
The following proposition collects some facts on the topological structure of models.
\begin{prop} \cite[6.2]{Abh_resolutionBook} \label{lem affines are open}
	Suppose \( \modelX \) is a model of \( K \) over \( D \). Then, the followings hold.
	\begin{enumerate}[i.]
		\item For \( R\in \modelX \) and \( A\in \affinesInX \), if \( A \subseteq R \)  then \( R\in \curlySOf{A} \). \label{item1 in lem affines are open}
		\item The set \( \set{\curlySOf{A}}{A\in \affinesInX} \) is a basis of \( \modelX \). \label{item2 in lem affines are open}
		\item The closure of an element \( R\in \modelX \) is
		\begin{equation*}
			\overline{\{R\}} = \set{S \in \modelX}{ S \subseteq R}.
		\end{equation*}
		In particular, the field \( K \) is the generic point of \( \modelX \). \label{item3 in lem affines are open}
	\end{enumerate}
\end{prop}

As an immediate consequence, we see that a model is a quasi-separated topological space. The following corollary is implicitly expressed in both \cite{ZariskiSamuel2} and \cite{Abh_resolutionBook}. It deserves a separate statement, as it will be useful for observing that the scheme structure of a model is separated. 
\begin{cor} \label{p. intersection of affines is affine of models}
	For a model \( \modelX \), the set \( \set{\curlySOf{A}}{A\in \affinesInX} \) is closed under finite intersections. Moreover, \( \modelX \) is a quasi-separated topological space; i.e. an intersection of two quasi-compact open sets is quasi-compact.
\end{cor}
\begin{proof}
	Fix \( A,\,B \in \affinesInX \). Then, the ring \( A[B]=B[A] \), the smallest subring of \( K \) that contains both \( A \) and \( B \), is also affine over \( D \). So, a ring \( R \) contains both \( A \) and \( B \) if and only if it contains \( A[B] \). Hence, from Proposition~\ref{lem affines are open}~\eqref{item1 in lem affines are open} the equality \( \curlySOf{A} \cap \curlySOf{B} = \curlySOf{A[B]} \) holds. By induction, we conclude that the given set is closed under finite intersections.
	
	Now, take two quasi-compact open subsets \( U \), \( V \) of \( \modelX \). So, they can be expressed as finite unions \( U=\bigcup_{i=1}^n \curlySOf{A_i} \) and \( V= \bigcup_{j=1}^m\curlySOf{B_j} \) where \( A_i \) and \( B_j \)'s are from \( \affinesInX \). 
	Then, we have
	\begin{equation*}
		U \cap V = \bigcup_{i=1}^n \bigcup_{j=1}^m \Big( \curlySOf{A_i}  \cap \curlySOf{B_j}   \Big) = \bigcup_{i=1}^n \bigcup_{j=1}^m \curlySOf{A_i[B_j]}  
	\end{equation*}
	which yields the desired quasi-compactness of \( U\cap V  \).
\end{proof}	

\begin{definition}  \label{def. domination}
	Let \( \modelX \) and \( \modelY \)  be premodels over \( D \).
	\begin{itemize}
		\item If every element in \( \modelX \) dominates an element of \( \modelY \),  then we say that	\( \modelX \) dominates \( \modelY \).  In this case, we denote \( \modelY \dominatedby \modelX \).
		\item We say that \( \modelX \) properly dominates \( \modelY \)  if  \( \modelY \dominatedby \modelX \) and every element of \( \modelY \)  is dominated by an element of \( \modelX \).
		\item If \( \modelY \)  is irredundant (e.g. a model over \( D \)), then \( \modelY \dominatedby \modelX \) if and only if every element \( R \) from \( \modelX \) dominates a unique element from \( \modelY \). We call this unique element the center of \( R \).
		\item Suppose \( \modelY \)  is irredundant and \( \modelY \dominatedby \modelX \). The map \( \centerFrom{\modelX}{\modelY} \) that sends an element of \( \modelX \) to its unique center is called the domination map from \( \modelX \) to \( \modelY \). 
	\end{itemize}
\end{definition}
Note that the domination relation is a preorder (i.e. reflexive and transitive relation) on the power set of \( \localDomains \); however, it is not antisymmetric. For instance, let \( \modelX = \curlySOf{D}\cup \curlySOf{D[x/y]}  \) as in Example~\ref{ex. non-irredundant premodel}. Then, \( \modelX \dominatedby \curlySOf{D} \) and \( \curlySOf{D} \dominatedby \modelX \) but \( \curlySOf{D} \subsetneq \modelX \).

The domination relation reverses the inclusion. For two models \( \modelX, \, \modelY \) over \( D \), if \( \modelX \subseteq \modelY \) as sets, then \( \modelY \dominatedby \modelX \). Moreover, in the case  \( \modelX \subseteq \modelY \), the domination map \( \centerFrom{\modelX}{\modelY}  \) is the inclusion map. 

\begin{prop} \cite[Lemmas 3 \& 5 in \S 17]{ZariskiSamuel2}
	Let \( \modelX \) be a premodel and \( \modelY \)  be a model over \( D \). Suppose \( \modelX \) dominates \( \modelY \). Then the domination map \( \centerFrom{\modelX}{\modelY} \colon \modelX \to \modelY \) is a continuous map. Moreover, if both \( \modelX \), \( \modelY \) are complete models, then \( \centerFrom{\modelX}{\modelY} \) is closed. 
\end{prop}

\subsection{Ideals on Models} \label{Subsec. Abhy ideals}
In \cite{Abh_resolutionBook}, \Abhy{} defines the notions of a preideal and an ideal on a given model over a Noetherian domain. Both definitions make sense without the Noetherianity assumption.

\begin{definition}  
	For a premodel \( \modelX \subseteq \localDomains \), a function 
	\begin{equation*}
		\abyideal : \modelX \longrightarrow \bigcup_{R \in \modelX} \{ \text{ideals of } R \}
	\end{equation*}
	is called a preideal on \( \modelX \) if for each \( R \in \modelX \) its image \( \abyideal(R) \) is an ideal of \( R \). 
\end{definition}
We write \( \abyideal R \)  instead of \( \abyideal (R) \).

\begin{notation} 
	Let \( \abyideal \)  be a preideal on a premodel \( \modelX \subseteq \localDomains \). For a subring \( A \) of \( K \) such that \( \curlySOf{A} \subseteq \modelX \), we denote the intersection \( \bigcap_{R \in \curlySOf{A}} \abyideal R \) by \( A\cap \abyideal \).
\end{notation}

\begin{definition} 
	Let \( \modelX \) be a model over an integral domain \( D \). A preideal \( \abyideal \)  on \( \modelX \) is called an ideal if the equality \( \abyideal R = \left( A \cap \abyideal \right) R \) holds for any \( A \in \affinesInX \) and \( R\in \curlySOf{A} \).	
\end{definition}
In \cite{Abh_resolutionBook}, the following proposition is proved under the assumption \( D \) is Noetherian. When \( D \) is Noetherian, a model \( \modelX \) over \( D \) is a Noetherian topological space with respect to Zariski topology. That is, all open subsets of \( \modelX \) are quasi-compact. However, the proof only uses the quasi-compactness of the open sets of the form \( \curlySOf{A} \) for \( A\in \affinesInX \). As the spectrum of any ring is quasi-compact, the proof is still valid when \( D \) is a non-Noetherian domain.

\begin{prop} \cite[6.4.3]{Abh_resolutionBook} \label{equiv def of aby ideal} 
	Let \( \modelX \) be a model over an integral domain \( D \). 
	Then, for a preideal \( \abyideal \)  on \( \modelX \) the following are equivalent. 
	\begin{enumerate}[(i)]
		\item \( \abyideal \)  is an ideal on \( \modelX \).
		\item For all \( S\in \modelX \), there exists \( A \in \affinesInX \)  such that \( S\in \curlySOf{A} \) and for every \( R\in \curlySOf{A} \) the equality  \( \abyideal R = ( A \cap \abyideal ) R \) holds.	\label{equiv def for aby ideal condition 2}
	\end{enumerate}
\end{prop}

By subtly modifying the preceding proposition, we see that an ideal \( \abyideal \)  on a model \( \modelX \) is uniquely determined by the ideals \( D_1 \cap \abyideal \ideal D_1 , \dotsc ,  D_n \cap \abyideal \ideal D_n \) where \( \modelX \) is defined by the integral domains \( D_1 , \dotsc ,  D_n \).
\begin{prop} \label{aby ideals are uniquely determined by ideals on affine pieces}
	Let \( \modelX \) be a model over an integral domain \( D \) and \( \abyideal \)  be a preideal on \( \modelX \). Assume that \( \modelX \) is defined by \( \big(D_{\eps}\big)_{\eps \in \Lambda } \subseteq \affinesInX \). Then, \( \abyideal \)  is an ideal on \( \modelX \) if and only if for every \( \eps \in \Lambda  \) and \( R\in \curlySOf{D_{\eps}} \) the equality  \( \abyideal R = (D_{\eps} \cap \abyideal) R \) holds.	
\end{prop}
\begin{proof}
	The necessity direction comes from the definition. For sufficiency, we note that for \( S\in \modelX \), choosing a \( D_{\eps} \) with \( S \in \curlySOf{D_{\eps}} \) satisfies Proposition~\ref{equiv def of aby ideal} \eqref{equiv def for aby ideal condition 2}.
\end{proof}
To put it in another way, for a model \( \modelX \) defined by \( \big\{ D_{\eps}\big\}_{\eps \in \Lambda } \subseteq \affinesInX \) and a family of ideals \( \big\{I_{\eps}\big\}_{\eps \in \Lambda } \) where \( I_{\eps} \ideal D_{\eps} \)  the proposition above tells us that the family \( \big\{I_{\eps}\big\} \) induces an ideal on \( \modelX \) if and only if for any \( R \in \modelX \) whenever \( R \) contains both \( D_{\eps} \), \( D_{\eps'} \), then the extended ideals \( I_{\eps}R \), \( I_{\eps'}R \) are equal. In particular, for an affine model \( \curlySOf{A} \), ideals on \( \curlySOf{A} \) can be identified with the actual ideals of \( A \).

\section{Models in Schemes Terminology} \label{Sec. models as schemes}
Recall that an abstract variety is an integral and separated scheme of finite type over an algebraically closed field. This section shows that models can be viewed as schemes, and they are generalizations of abstract varieties. The terminology related to schemes can be found in \citet{Hartshorne2010algebraic}. Occasionally, the theorems in~\cite{Hartshorne2010algebraic} are restricted to Noetherian rings. But we work in the framework of models over an integral domain \( D \) without assuming Noetherianity. For the employed theorems that need a more general setting, we refer the reader to~\citet{GortzAlgGeo2020}.

\begin{notation} \label{sheaf structure of a model}
	For a premodel \( \modelX \), we will denote by \( {\sheafO}_{\modelX} \) the sheaf of rings defined by 
	\begin{equation*}
		{\sheafO}_{\modelX} (U)  \coloneqq  \bigcap_{R \in \modelX} R  
	\end{equation*}
	where \( U  \) is an open subset of \( \modelX \). We take the restriction maps as inclusions (or the zero map when the codomain is the zero ring). Note that by convention, the empty intersection gives the zero ring. 
\end{notation}
A premodel \( \modelX \) can be viewed as a locally ringed space via the sheaf \( {\sheafO}_{\modelX} \). Moreover, for \( R\in \modelX \), the stalk \( \sheafO_{\modelX, R} \) at \( R \) is the local ring \( R \). We prove these remarks.
\begin{lemma} \label{premodel is a locally ringed space}
	Let \( \modelX \) be a premodel. Then \( (\modelX, {\sheafO}_{\modelX} ) \) is a locally ringed space.
\end{lemma}
\begin{proof}
	First, we argue that \( {\sheafO}_{\modelX} \) is a sheaf. Note that for non-empty open sets \( V \subseteq U  \), we clearly have \( {\sheafO}_{\modelX} (U) \subseteq {\sheafO}_{\modelX} (V)  \). So, taking the restriction maps as inclusions makes sense. Since the restriction maps are inclusions, it is evident that the sheaf axioms are satisfied. 	
	Now, fix \( R\in \modelX \). We claim that the stalk \( \sheafO_{\modelX, R} \) is the local ring \( R \). For any open set \( U \) that contains \( R \), the ring \( {\sheafO}_{\modelX} (U) \) is a subring of \( R \). Therefore, the stalk \( \sheafO_{\modelX, R} \) is also a subring of \( R \). Note that for any \( r\in R \), since \( U_r = \set{S\in \modelX}{S \ni r} \) is an open set that contains \( R \), the integral domain \( {\sheafO}_{\modelX} (U_r) \) is a subring of the stalk \( \sheafO_{\modelX, R} \).
	Therefore, \( \sheafO_{\modelX, R} \) cannot be a proper subring of \( R \), because otherwise there would be \( r \in R \setminus \sheafO_{\modelX, R} \) causing \( {\sheafO}_{\modelX} (U_r) \nsubseteq \sheafO_{\modelX, R} \).
	So all stalks are local rings and hence \( (\modelX, {\sheafO}_{\modelX} ) \) is a locally ringed space.
\end{proof}

Using that a model is a union of affine models and is irredundant, we now conclude any model can be considered as a scheme. Despite the convention of denoting a scheme solely by its underlying topological space, we mostly maintain the tuple notation to clearly differentiate between the model itself and the corresponding ringed space. Still, we occasionally switch back to the conventional notation when it eases the notation without obscuring our perspective. 
\begin{prop}
	Let \( \modelX \) be a model over \( D \). Then, \( (\modelX , {\sheafO}_{\modelX} ) \) is a quasi-compact, quasi-separated and an integral scheme. 
\end{prop}

\begin{proof}
	We first note that being quasi-compact and quasi-separated are topological properties, and Corollary~\ref{p. intersection of affines is affine of models} shows that models are quasi-separated. Moreover, \( \modelX \) can be written as a finite union of quasi-compact open sets since the spectrum of any ring is quasi-compact and \( \curlySOf{A} \) is open in \( \modelX \) for any \( A\in \affinesInX \);  thus, \( \modelX \) is quasi-compact. 
	
	As established in Lemma~\ref{premodel is a locally ringed space},  \( (\modelX, {\sheafO}_{\modelX} ) \) is a locally ringed space. Additionally, for any open set \( U \), by definition \( \sheafO(U) = \cap_{R\in U} R \). First, we assume \( \modelX \) is an affine model, say \( \modelX = \curlySOf{A} \) for an integral domain \( A \). Therefore, when we identify \( \curlySOf{A} \) with \( \Spec (A) \), the ring \( \sheafO(U)  \) is the intersection of the localizations at primes contained in \( U \). Hence, \( (\modelX, {\sheafO}_{\modelX} ) \) is the affine scheme of \( A \).
	
	For the general case, write \( \modelX = \bigcup \curlySOf{D_{\eps}} \) where \( D_{\eps} \) are affine rings over \( D \). We set \( U_{\eps} = \curlySOf{D_{\eps}} \) and consider the restricted sheaf \( {\sheafO}_{\modelX} \restriction_{U_{\eps}} \). From the affine case we know that the restricted ringed spaces \( (U_{\eps} , {\sheafO}_{\modelX}{\restriction_{U_{\eps}}}) \) are the affine schemes of the rings \( D_{\eps} \)'s, 
	and hence, \( (\modelX, {\sheafO}_{\modelX} ) \) is a scheme. Moreover, for any open \( U\subseteq \modelX \), the ring \( {\sheafO}_{\modelX} (U) \) is an integral domain as it is an intersection of integral domains. This demonstrates \( (\modelX, {\sheafO}_{\modelX} ) \)  is an integral scheme.
\end{proof}

\begin{remark}
	We can also construct \( ( \modelX, {\sheafO}_{\modelX} ) \) as a gluing of the affine schemes \( U_{\eps} = \curlySOf{ D_{\eps} } \) where \( \{D_{\eps}\}_{\eps=1}^n \) is a family that defines \( \modelX \).
	The proof of Corollary~\ref{p. intersection of affines is affine of models} shows that for any \( \eps, \eps' \) the intersection \( U_{\eps} \cap U_{\eps'}= \curlySOf{D_{\eps} + D_{\eps'}}  \) is open in \( \modelX \). So, the collection of \( U_{\eps} \cap U_{\eps'} \)'s together with identity functions between them forms a gluing datum. 
	Since \( (U_{\eps})_{\eps =1}^n \) covers \( \modelX \) and we glue them along identity maps, the obtained glued scheme is indeed \( ( \modelX, {\sheafO}_{\modelX} ) \). Moreover, by gluing the inclusion maps \( D \hookrightarrow D_{\eps} \), more precisely by gluing the scheme morphisms \( U_{\eps} \to \curlySOf{D} \) that correspond to the inclusions \( D \subseteq D_{\eps} \), we obtain a scheme morphism from \( ( \modelX, {\sheafO}_{\modelX} ) \) to \( (\curlySOf{D} , \sheafO_{D}) \). The continuous map on the underlying topological spaces is the center map from \( \modelX \) to \( \curlySOf{D} \).  
\end{remark}
In particular, the following theorem shows that a model over an algebraically closed field is an abstract variety.
\begin{thm} \label{model X is seperated D scheme}
	Let \( \modelX \) be a model over \( D \). Then, the scheme \( \modelX \) is a separated \( D \)-scheme of finite type. Moreover, the scheme \( \modelX \) is proper over \( D \) if and only if \( \modelX \) is a complete model.
\end{thm} 
\begin{proof} 
	Say \( \modelX = \cup_{\eps \in \Lambda } \curlySOf{D_{\eps}} \) where \( \Lambda \) is a finite index set and \( D_{\eps}  \)'s are in \( \affinesInX \). In particular, for any \( \eps \), the ring \( D_{\eps} \) is a finitely generated \( D \)-algebra which implies the scheme \( \curlySOf{D_{\eps}} \) is of finite type over  \( D \). Therefore, \( ( \modelX, {\sheafO}_{\modelX} ) \) is of locally finite type over \( D \) as it has an affine cover of finite type over \( D \). Combining with the fact that it is quasi-compact, we conclude that \( ( \modelX, {\sheafO}_{\modelX} ) \) is a \( D \)-scheme of finite type.  
	
	The scheme morphism from \( ( \modelX, {\sheafO}_{\modelX} ) \) to the affine scheme \( \curlySOf{D} \) is quasi-separated since so are both schemes. 
	Moreover, a valuation domain can dominate at most one element of \( \modelX \) as models are irredundant premodels. Therefore, Valuative criterion - general version (\cite[Theorem~15.8]{GortzAlgGeo2020}) yields that the morphism \( \modelX \to \curlySOf{D} \) is separated. Additionally, from by the same criterion, we conclude that  \( \modelX \to \curlySOf{D} \) is proper if and only if any valuation domain that contains \( D \) must dominate an element of \( \modelX \). The latter condition is equivalent that the model \( \modelX \) is complete.
\end{proof}
Conversely, for a given integral scheme \( ( \modelX, {\sheafO}_{\modelX} ) \)  that is separated and of finite type over \( D \), its underlying topological space can be viewed as a model over \( D \). 
\begin{prop} \label{t. integral scheme is model}
	Let \( ( \modelX, {\sheafO}_{\modelX} ) \) be an integral scheme that is separated and of finite type over \( D \). Assume that its function field is \( K \). Then, there is a model \( \modelY \)  of \( K \) over \( D \) such that the corresponding scheme \( (\modelY, {\sheafO}_{\modelY} ) \) is isomorphic to \( ( \modelX, {\sheafO}_{\modelX} ) \).
\end{prop}
\begin{proof}
	The reason we do not directly use  \( \modelX \) is that it does not have to be a subset of \( \localDomains \). Since \( ( \modelX, {\sheafO}_{\modelX} ) \) is an integral \( D \)-scheme, we may find an affine cover \( {U_i} \cong \Spec A_i  \) such that for each \( i \)  the integral domain \( A_i \) contains \( D \) and has fraction field \( K \). Moreover, as \( \modelX \) is quasi-compact, we can choose a finite cover, say \( U_1 , \dotsc ,  U_n \). Then taking \( \modelY = \bigcup_{i=1}^n \curlySOf{A_i} \) works. 
\end{proof}

Predictably, the set of affine open subschemes of \( \modelX \) corresponds to \( \affinesInX \).
\begin{lemma} \label{affine opens are affines in model}
	Let \( \modelX \) be a model of \( K \) over \( D \) and \( U \) be an affine open subscheme of \( ( \modelX, {\sheafO}_{\modelX} ) \). Then the ring \( {\sheafO}_{\modelX} (U) \) is in \( \affinesInX \).
\end{lemma}
\begin{proof}
	Say \( A= {\sheafO}_{\modelX} (U) \), which means \( U \cong \curlySOf{A}  \) and \( K \) is the fraction field of \( A \). Since \( U \) is an affine open subscheme, we have an open immersion \( U \hookrightarrow \modelX \). Recall that immersions are of locally finite type. Therefore, as \( \modelX \) is quasi-compact, \( U \hookrightarrow \modelX \) is of finite type. Now, by composing it with \( \modelX \to \curlySOf{D} \), we see that the scheme morphism \( U \to \curlySOf{D} \) is of finite type as this property is preserved under compositions. In particular, as a \( D \)-algebra, \( A= {\sheafO}_{\modelX} (U) \) is finitely generated and hence we obtain \( A \in \affinesInX \) as claimed. 
\end{proof}

\begin{prop} \label{p. correspondance of ab and qCoh ideals}
	Let \( \modelX \) be a model over \( D \). Ideals on the model \( \modelX \) correspond to the quasi-coherent sheaves of ideals on the scheme \( ( \modelX, {\sheafO}_{\modelX} ) \). 
\end{prop}
\begin{proof}
	First, suppose \( \modelX = \curlySOf{A} \) for an integral domain \( A \). An \( {\sheafO}_{\modelX} \)-module \( \sheafideal \) is a quasi-coherent sheaf of ideals on the affine scheme corresponding to \( A \) if and only if \( \sheafideal = \tilde{I} \) for some ideal \( I \) of \( A \) where \( \tilde{I} \) is the unique sheaf satisfying \( \tilde{I} \big({A[1/a]}\big) = I A[1/a] \) for all nonzero \( a\in A \).  In other words, any quasi-coherent sheaf of ideals on \( \curlySOf{A} \) is the sheaf whose stalks are extensions of an ideal \( I \ideal A \). On the other hand, by Proposition~\ref{aby ideals are uniquely determined by ideals on affine pieces} we know that the ideals on the affine model \( \curlySOf{A} \) correspond to the ideals of \( A \) and hence to the quasi-coherent sheaf of ideals.
	
	For the general case, we note that an \( {\sheafO}_{\modelX} \)-module \( \sheafideal \) is a quasi-coherent sheaf of ideals if and only if so is the restriction \( \sheafideal \restriction_U\) for any affine open subscheme \( U \). Hence, the general case follows from Lemma~\ref{affine opens are affines in model}, which establishes the identification of affine open subschemes of \( \modelX \) and \( \affinesInX \). 
\end{proof}

\section{Intersecting Nagata Extensions Obtained from a Model} \label{Sec. Dstar of models}

We begin by recalling the definition of the Nagata extension of a given ring \( R \). More information on Nagata extensions can be found in \citet{gilmer1972multiplicative}. 
For the rest of this paper, $t$ will denote an indeterminate. 

\begin{definition}
	A polynomial in \( R[t] \) is called primitive if its coefficients generate the unit ideal \( R\).
\end{definition}
There is ambiguity in the literature concerning the terminology of primitive polynomials, particularly when the ring is a GCD (Greatest Common Divisor) domain. In some sources (e.g. \cite{ZariskiSamuel_Volume1_RenewedEdition}), a polynomial \( f\in R[t] \) is called primitive if the greatest common divisor of the coefficients of \( f \) is unit. This does not always coincide with the definition we adopted. For instance, when \( R \) is taken as the GCD domain \( \complexC[x,y] \), the greatest common divisor of the coefficients of the polynomial \( f(t)=xt+y \in R[t] \) is a unit; however, we do not consider \( f \) as primitive since \( \gendby{x,\,y} \) is a proper ideal. 

For any ring \( R \), the set of primitive polynomials from \( R[t] \) is a multiplicative set. So, the following definition makes sense.
\begin{definition} 
	Let \( R \) be a ring and \( N \) denote the set of all primitive polynomials from \( R[t] \). Then the localization \( N\inverse R[t] \) is called the Nagata extension (or the Nagata	function ring) of \( R \) and is denoted by \( R(t) \).
\end{definition}
Note that the Nagata extension of a field \( K \) is the usual field extension \( K(t) \). 

\begin{definition} \label{notation Dstar from models}
	For a given model \( \modelX \) of \( K \) over \( D \), we set 
	\begin{equation*}
		\dstarof[\modelX] = \bigcap_{R \in \modelX} R(t) .
	\end{equation*}
	For readability, when \( \modelX \) is clear from the context, we omit \( \modelX \) and write \( \dstar \).
\end{definition}
Clearly, \( \dstar \) is an integral domain and its fraction field is \( K(t) \). 

\begin{prop} \label{p. equiv def for Dstar}
	Suppose \( \modelX \) is a model over \( D \) and \( A \in \affinesInX \). Then, 
	\begin{equation*}
		I = \bigcap_{R \in \curlySOf{A}} IR(t) 
	\end{equation*}
	for any ideal \( I\ideal A(t) \). Moreover, 
	\begin{equation} \label{eq. Dstar cover dependent definition}
		\dstar = \bigcap_{\eps \in \Lambda } D_{\eps} (t)
	\end{equation}
	where \( \big\{ D_{\eps} \big\}_{\eps \in \Lambda } \subseteq \affinesInX \) is a family that defines \( \modelX \).
\end{prop}
\begin{proof}
	Note that \( I \subseteq \bigcap  IR(t) \) trivially follows from the fact that for any \( R\in \curlySOf{A} \), we have \( A(t) \subseteq R(t) \). On the other hand, the inequality 
	\begin{equation*}
		\bigcap_{R \in \curlySOf{A}} IR(t) \ 
		\subseteq \bigcap_{M \in \Max A} IA_M (t) 
		\ 	= \bigcap_{M' \in \Max  A(t) } I {A (t) }_{M'}   
		= I  
	\end{equation*}
	gives the reverse inclusion. By taking \( I \) as a unit ideal, we obtain that \(A(t) = \bigcap_{R \in \curlySOf{A}} R(t) \) and hence 
	\begin{equation*}
		\dstar 
		= \bigcap_{\eps \in \Lambda } \bigcap_{R \in \curlySOf{D_{\eps}}} R(t) 
		= \bigcap D_{\eps} (t)
	\end{equation*}
	as needed.
\end{proof}
Most of the time, we take Equation~\eqref{eq. Dstar cover dependent definition} as the definition for \( \dstar \).

For a projective model over \( D \), the intersection \( \bigcap_{R\in \modelX} R\) is \( D \) itself. However, in general \( \dstar \) properly contains \( D(t) \).
\begin{example}
	Let \( D=\complexC[x,y] \) and \( \modelX  \) be the projective model defined by \( x ,\,y \). Then \( \dstar  = D [x/y ] (t) \cap D [y/x ] (t) \). Note that
	\begin{align*}
		\frac{y}{xt +y} &\in \dstar \\
		\shortintertext{since }
		\frac{y}{xt +y} = \frac{1}{(x/y)t +1}  \in D [x/y ] (t) \quad & \text{and} \quad 
		\frac{y}{xt +y} = \frac{y/x}{(y/x) t +1} \in  D [y/x ] (t) 
	\end{align*}
	but \( \frac{y}{xt +y} \notin D(t) \). Thus,	\( D(t) \subsetneq  \dstar\).
\end{example}

\begin{prop} 
	Let \( \modelX, \modelY \) be models of \( K \) over \( D \). If \( \modelX \) is dominated by \( \modelY \), then \( \dstarof[\modelX]  \)  is subring of \(  \dstarof[\modelY] \). 
\end{prop}
\begin{proof}
	Assume \( \modelX \dominatedby \modelY \) and fix \( R\in \modelY \). Then, \( R \) has a center in \( \modelX \). Moreover, its center is a subring of \( R \) and hence we have the same relation for their Nagata extensions. Therefore, \( \dstarof[\modelX] \) is a subring of \( R(t) \) and so intersecting the rings \( R(t) \), where \( R \) runs through the set \( \modelY \), demonstrates that the ring  \( \dstarof[\modelX] \) is contained in \( \dstarof[\modelY] \).
\end{proof}

\subsection{Construction arising from a projective model} \label{Subsec. construction Dstar of models}
We focus on projective models.

\begin{lemma} \label{lem technical for localization projmodel} 
	Let \( \modelX \) be a projective model over \( D \) defined by \( \{ x_1,\dotsc, x_n \}\subseteq K\setminus \{ 0\} \). Let \( J \) be the \( D \)-submodule of \( K \) generated by \( \{ x_1,\dotsc, x_n \} \). Then, for any nonzero \( a \in J \),  the polynomial $\theta/a \in D[Ja\inverse][t]$ is primitive where
	\begin{equation*}
		\theta = \sum_{i=1}^{n} x_{i} t^{{i-1}}  
	\end{equation*}
	and hence $a/ \theta$ is an element of \( \dstar \).
\end{lemma}
\begin{proof}
	It is enough to assume $a\in \{ x_1,\dotsc, x_n \}$. 
	For $\eps \in \{1 , \dotsc ,  n\}$, we let \( D_{\eps}= D[J{x_{\eps}}\inverse ] \), so $\{D_{\eps} (t) \}_{\eps =1}^n$ is an open cover of \( \modelX \). Now, fix an arbitrary $\eps   \in \{ 1, \dotsc , n\}$. Note that $\theta / x_{\eps} \in D_{\eps} [t]$ and the coefficient of $t^{{\eps}-1}$ is $ x_{\eps}/x_{\eps}=1$. Therefore, $\theta / x_{\eps} $ is a primitive polynomial in \( D_{\eps} [t] \) and hence
	\begin{equation*}
		\frac{a}{\theta} = \frac{a/x_{\eps} }{\theta/x_{\eps} } 
	\end{equation*}
	is an element of \( D_{\eps} (t) \). Thus, \( a/\theta \in \cap D_{\eps} (t) = \dstar \).
\end{proof}

\begin{prop} \label{localization for proj models} 
	Let \( \modelX \) be a projective model, say \( \modelX= \projModel (D;J) \) where \( J \) is a finitely generated \( D \)-submodule of \( K \). Then, for any nonzero \( a \in J \), the ring $D[Ja\inverse](t)$ is a localization of \( \dstar \). 
\end{prop}
\begin{proof}
	Fix a nonzero \( a \in J \) and let \( N \) be the multiplicative subset generated by \( a/\theta \in \dstar \) where \( \theta = x_1 + \dotsb + x_n t^{n-1} \) is as in Lemma~\ref{lem technical for localization projmodel} for some generator set \( \{x_1 , \dotsc ,  x_n\} \) of \( J \). Since \( a/\theta \) is a unit element of \( D[Ja\inverse](t) \), the localization \( N\inverse \dstar = \dstar [\theta / a] \) is contained in $D[J a\inverse](t)$. Moreover, for any \( b\in J \)
	\begin{equation*}
		\frac{b}{a} = \frac{b}{\theta} \cdot \frac{\theta}{a} \in \dstar [\theta / a]
	\end{equation*}
	since \( {\theta}/{a} \) is in \( \dstar \) by Lemma~\ref{lem technical for localization projmodel}. Therefore, we get
	\begin{equation} \label{eqLemmaLocalizationForProjModels}
		D[J a\inverse] [t] \subseteq N\inverse \dstar \subseteq D[J a\inverse] (t) .
	\end{equation} 
	Now, let \( S \) denote the multiplicative subset of \( D[J a\inverse] [t] \) consisting of the primitive polynomials. Then, \( D[J a\inverse](t) = S\inverse \Big( D[J a\inverse] [t] \Big)  \). Localizing the inequality in Equation~\eqref{eqLemmaLocalizationForProjModels} at the multiplicative set \( S \) gives that $D[J a\inverse] (t) = S\inverse \left( N\inverse \dstar \right)$. So we conclude that \( D_{\eps} (t) \) is a localization of \( \dstar \).
\end{proof}

\begin{thm} \label{t. an ideal of ProjModel-Dstar survives in some epsilon}
	Assume \( \modelX \) is a projective model defined by a finite subset \( \{ x_1,\dotsc, x_n \} \) of \( K\setminus \{ 0\} \). For each \( \eps \) from \( 1 \) to \( n \), we let \( D_{\eps}  \) denote  \( D[x_1/ x_{\eps} ,\dotsc, x_n/ x_{\eps} ] \). Then, for any proper ideal \( Q \) of \( \dstar \), there exists some \( \eps \) such that \( Q D_{\eps} (t) \) is proper in \( D_{\eps} (t) \). 
\end{thm}  
\begin{proof}
	We have $\modelX = \bigcup_{\eps= 1}^{n}  \curlySOf{D_{\eps}}$. Fix an ideal $Q\ideal \dstar $ such that for all $\eps \in \{1 , \dotsc ,  n \}$ the extended ideal \( Q D_{\eps} (t) \) is the unit ideal \( D_{\eps} (t) \). We will show that \( Q \) must be the unit idea. By assumption, for each \( i \) from \( 1 \) to \( n \), the ideal \( Q \) contains an element $\alpha_i \in \dstar$ that is invertible in \( D_{i} (t) \) since from Proposition~\ref{localization for proj models} \( D_{i} (t) \) is a localization of \( \dstar \). Therefore, as all $\alpha_1 , \dotsc ,  \alpha_n $ are from $\dstar =\bigcap D_{\eps} (t) $, for each ordered pair $(i,\eps)$, there exists polynomials $f_{i,{\eps}}, \, g_{i,{\eps}} \in D_{\eps} [t]$ such that 
	\begin{equation*}
		\alpha_i = \frac{f_{i,{\eps}}}{g_{i,{\eps}}}
	\end{equation*}
	where $g_{i,{\eps}}$ is primitive in \( D_{\eps} [t] \) and whenever $\eps = i$, both $f_{i,{\eps}}, \, g_{i,{\eps}} $ are primitive in \( D_{\eps} [t] \). Let 
	\begin{align} 
		c\, &= 1 +\, \sum_{ i, \eps = 1 , \dotsc ,  n } \deg_t f_{i,{\eps }} \,+\,  \sum_{ i, \eps = 1 , \dotsc ,  n } \deg_t g_{i,{\eps }} \in \naturalN \nonumber
		\\
		\shortintertext{and }
		\varphi &= \sum_{i=1}^{n} \alpha_i \, t^{(i-1)c} \in \dstar . \label{eq for phi}
	\end{align}
	We claim that \( \varphi \)  is a unit in all $ D_{\eps} (t)$'s, hence in \( \dstar \) as well. To argue this fix an \( \eps \). By substituting $\alpha_i$'s, we can write the Equation (\ref{eq for phi}) as 
		\begin{align}
		\varphi &= \frac{f_{1,{\eps }}}{g_{1,{\eps }}} + \frac{f_{2,{\eps }}}{g_{2,{\eps }}} t^{c} + \cdots + \frac{f_{n,{\eps }}}{g_{n,{\eps  }}} t^{(n-1)c}  \nonumber \\
		&= \frac{f_{1,{\eps }} \cdot g_{2,{\eps }} \cdots g_{n,{\eps }} +
			 \cdots +  f_{n,{\eps }}\cdot g_{1,{\eps }}  \cdots g_{n-1,{\eps }} \, t^{(n-1)c} }{ g_{1,{\eps }}  \cdots g_{n,{\eps }} }   \label{eq line2}
	\end{align}

	Notice that in Equation~\eqref{eq line2} both the numerator 
	\begin{equation} \label{eq numerator}
		\sum_{i=1}^{n} \bigg(  f_{i,{\eps }} \Big(  \prod_{\substack{k=1 \\ k\neq i}}^n g_{k,{\eps }} \Big) t^{(i-1)c}  \bigg) 
	\end{equation}
	and the denominator 
	are polynomials in \( D_{\eps} [t] \). Moreover, the denominator is primitive in \( D_{\eps} [t] \), since all $g_{1,{\eps }} ,\dotsc , g_{n,{\eps }}$ are primitive as well. Furthermore, $c$ is chosen large enough that the content ideal of the numerator, i.e. the ideal generated by its coefficients, is generated by the content ideal of the polynomials that appear as summands in Equation (\ref{eq numerator}). In particular, the content ideal of the numerator contains the content ideal of the polynomial $f_{\eps,{\eps}}  \prod_{k\neq \eps} g_{k,{\eps}}$, which is primitive as it is a product of primitive polynomials. Therefore, the polynomial in Equation (\ref{eq numerator}) must be primitive since its content ideal contains the content ideal of a primitive polynomial.	Hence, Equation (\ref{eq line2}) gives a representation for \( \varphi \)  as a quotient of two primitive polynomial from \( D_{\eps} [t] \). Thus, \( \varphi \) is a unit in \( D_{\eps} (t) \). This demonstrates that $\varphi \in Q$ is a unit in \( \dstar \). So, \( Q \) cannot be proper. 
\end{proof}
As an immediate consequence of Proposition~\ref{localization for proj models} and Theorem~\ref{t. an ideal of ProjModel-Dstar survives in some epsilon}, we have the following.
\begin{cor} \label{proj-model max Dstar comes from Depsilons}
	Assume \( \modelX \) is a projective model defined by nonzero elements \( x_1,\dotsc , x_n  \) of \( K \). We set $D_{\eps} = D[x_1/ x_{\eps} ,\dotsc, x_n/ x_{\eps } ]$ for $\eps  \in \{1 , \dotsc ,  n\}$. Then, any maximal ideal of \( \dstar \) is a contraction of a maximal ideal of \( D_{\eps} (t) \) for some \( \eps \) in \( \{1 , \dotsc ,  n\} \).
\end{cor}

\begin{proof}
	Let $M \ideal \dstar$ be maximal. Then, from Theorem~\ref{t. an ideal of ProjModel-Dstar survives in some epsilon}, for some \( \eps \), the extension $MD_{\eps} (t)$ is proper. Moreover, by Proposition~\ref{localization for proj models}, \( D_{\eps} (t) \) is a localization of \( \dstar \), hence $MD_{\eps} (t)$ must be a maximal ideal. Therefore, we get $M = MD_{\eps} (t) \cap \dstar$.
\end{proof}

\begin{cor} \label{p. ideal on Dstar is intersection over model}
	Assume \( \modelX \) is a projective model. For any $I \ideal \dstar$,  we have 
	\begin{equation*}
		I = \bigcap_{R \in \modelX} I R(t) .
	\end{equation*}
\end{cor}
\begin{proof}
	Say \( \modelX \) is defined by \( \{ x_1,\dotsc, x_n \}\subseteq K\setminus \{ 0\} \) and for each \( \eps \) from \( \{1 , \dotsc ,  n\} \), we let $D_{\eps} = D[x_1/ x_{\eps } ,\dotsc, x_n/ x_{\eps } ]$. We trivially have \( 	I \subseteq \bigcap R(t) \). On the other hand,
	\begin{equation*}
		I = \bigcap_{M \in \Max \dstar } I \dstar_{M} 
		\supseteq  \bigcap_{\eps =1}^n \ \bigcap_{M \in \Max D_{\eps} (t)} I D_{\eps} (t)_{M} 
		=\  \bigcap_{\eps =1}^n I D_{\eps} (t)
	\end{equation*}
	where the first and last equalities hold in general for integral domains and the middle inequality comes from Corollary~\ref{proj-model max Dstar comes from Depsilons}. Then, Proposition~\ref{p. equiv def for Dstar} gives the desired equality.
\end{proof}

Given a projective model \( \modelX \), we consider the mapping that sends a prime ideal \( P \ideal \dstar \)  to the ideal  \( D_{\eps} \cap P  \) for a fixed \( \eps \)  such that  \( P \) survives in \( D_{\eps} (t) \). Considering the prime ideals as the localizations of their corresponding rings, we get a map from \( \curlySOf{\dstar}  \) to \( \modelX \). Moreover, combined with Proposition~\ref{localization for proj models}, Corollary~\ref{proj-model max Dstar comes from Depsilons} shows that it sends closed points to the closed points of \( \modelX \). The following lemma implicitly demonstrates that the described map is, in fact, the domination map from \( \curlySOf{\dstar}  \) to \( \modelX \).
\begin{lemma}
	Let \( \modelX  \) be a projective model. Then \( \curlySOf{\dstar} \) properly dominates \( \modelX \).
\end{lemma}
\begin{proof}
	Fix $R \in \curlySOf{\dstar}$ and write \( R=\dstar_{P} \) for a prime ideal \( P \) of \( \dstar \). Then, Theorem~\ref{t. an ideal of ProjModel-Dstar survives in some epsilon} gives the existence of an \( \eps \)  such that  \( P D_{\eps} (t)  \) is strictly contained in \( D_{\eps} (t) \) and from Proposition~\ref{localization for proj models} we must have \( R = D_{\eps} (t)_{PD_{\eps} (t)} \). Let \( Q = D_{\eps} \cap P \) and \( S  \in \modelX \) be the localization of \(  D_{\eps} \) at the prime \( Q \). Since \( Q \subseteq P \), it easily follows that \( R \) dominates \( S \). Thus, \( \modelX \dominatedby \curlySOf{\dstar} \). The fact that for any \( R \in \modelX \) its Nagata extension \( R(t) \) is in \( \curlySOf{\dstar} \) demonstrates that the domination is proper.
\end{proof}

Since  \( \modelX \dominatedby \curlySOf{\dstar} \), the domination map \( \centermap : \curlySOf{\dstar} \to \modelX \) is continuous. In fact, it induces a morphism of schemes from the affine scheme $\curlySOf{\dstar}$ to \( ( \modelX , {\sheafO}_{\modelX} ) \).
\begin{thm} \label{thm. flat morphism from Spec Dstar to projective X}
	Suppose \( \modelX  \) is a projective model. Then $\centermap :  \curlySOf{\dstar} \to \modelX$ is a morphism of schemes that is faithfully flat.
\end{thm}
\begin{proof}
	Note that for any open set \( U \subseteq \modelX \), we have \( {\sheafO}_{\modelX} (U) = \bigcap_{R\in U} R \) and so \( {\sheafO}_{\modelX} (U) \subseteq \sheafO_{\dstar} ({\centermap}\inverse (U)) \) as all \(  \centermap(R^{\ast})\) are contained in \( R^{\ast} \). Hence, \( \centermap \) induces a morphism of sheaves \( {\sheafO}_{\modelX} \to {\centermap}_{\ast} \sheafO_{\dstar} \) where \( {\centermap}_{\ast} \sheafO_{\dstar} \) is the direct image sheaf.	Moreover, for any \( R^{\ast} \in \curlySOf{\dstar} \) the induced map at stalks is the inclusion \( \centermap(R^{\ast}) \hookrightarrow R^{\ast} \) and hence a local map. This completes the proof that \( \centermap \) gives a morphism of schemes. 
	
	To argue the faithful flatness, we need to show that \( R^{\ast} \) is faithfully flat over \( \centermap(R^{\ast}) \) for any \( R^{\ast} \in \curlySOf{\dstar} \). We fix \( R^{\ast} \in \curlySOf{\dstar} \) and let \( R = \centermap(R^{\ast}) \). As \( R \subseteq R^{\ast} \), it is enough to check for any ideal \( I \) of \( R \), its extension \( I R^{\ast} \) is proper in \( R^{\ast} \). This clearly holds since \( R \) is dominated by \( R^{\ast} \). 
\end{proof}

\subsection{Relevant ideals} \label{Subsec. rel Ideals of Dstar}

We focus on the ideals of \( R(t) \) that are extensions of the ideals of a ring \( R \). 
\begin{definition}
	An ideal \( I\ideal  R(t) \) is called relevant if \( I \) can be generated by elements from \( R \); in other words $I = ( I \cap R )  R(t)$.
\end{definition}
All maximal ideals of \( R(t) \) are relevant. However, unless \( R \) is an arithmetical ring, not all ideals of \( R(t) \) are relevant (see \cite{anderson1976multiplication}).
\begin{example}
	For the ring \( \complexC[x,y] \), the principal ideal \( \gendby{xt+y }\)  is not relevant. 
\end{example}

We will implicitly adopt any of the following equivalent conditions as our definition. 
\begin{prop}
	For an ideal \( I \) of \( R(t) \), the followings are equivalent.
	\begin{enumerate} [(a)]
		\item \( I \) is relevant.  \label{itemC in homLemma}
		\item \( I\cap R[t] \) is homogeneous (of degree $0$). \label{itemB in homLemma}
		\item \( I \) is an extension of a homogeneous ideal of \( R[t] \). \label{itemD in homLemma}
	\end{enumerate}
\end{prop}
\begin{proof}
	The equivalence of \eqref{itemC in homLemma}  and \eqref{itemB in homLemma} comes from the observation that as elements of \( R(t) \), a homogeneous element of \( R[t] \) is an associate of an element of \( R \). Moreover, the implication \eqref{itemC in homLemma} implies \eqref{itemD in homLemma} is trivial since assuming~\eqref{itemC in homLemma} gives that $(I\cap R)R[t]$ is homogeneous and extends to \( I \). Finally, assume~\eqref{itemD in homLemma}; that is \( I \) is an extension of a homogeneous ideal from \( R[t] \), say $I'$.  As a homogeneous ideal, $I'$ has a set of generators consisting of homogeneous elements from \( R[t] \).  
	The same set also generates \( I \) in $R (t)$ and furthermore, using the fact that $t $ is unit in \( R(t) \), we can switch the generators by their associate elements from \( R \). Hence, \( I \) can be generated by elements from~\( R \).
\end{proof}

We now extend this definition to the ideals of \( \dstar \) where \( \dstar \) is obtained from a model \( \modelX \) as in Definition~\ref{notation Dstar from models}.

\begin{definition}
	An ideal $I\ideal \dstar$ is called relevant if its extension $I R (t)$ is relevant for all \( R\in \modelX \).
\end{definition}

\begin{example}
	Let \( D=\complexC[x,y] \) and \( \modelX \) be the projective model defined by \( x,\,y \). Then as ideals of \( \dstarof[\modelX] \), the ideal \( \gendby{xt+y }\)  is relevant, but \( \gendby{xt^2+y } \) is not. Moreover,  \( \gendby{xt+y , \, \frac{x}{xt+y}} \) is also relevant.
\end{example}
In similarity with Proposition~\ref{p. equiv def for Dstar}, instead of checking $IR(t)$ for all \( R \in \modelX \), one can also check that \( IA(t) \) is relevant for all \( A\in \affinesInX \).
\begin{prop} 
	Let \( \{ D_{\eps}\}_{\eps  \in \Lambda }\subseteq \affinesInX \) define \( \modelX \). For an ideal \( I \) of \( \dstar \) the following are equivalent.
	\begin{enumerate}[(a)]
		\item $ID_{\eps} (t) $ is relevant, for any \( \eps \). \label{enum1}
		\item $IR(t)$ is relevant, for any \( R\in \modelX \). \label{enum2}
		\item \( IA(t) \) is relevant, for any affine ring \( A \) of \( D \) satisfying \( \curlySOf{A} \subseteq \modelX \). \label{enum3}
	\end{enumerate}
\end{prop}
\begin{proof}
	Assume \eqref{enum1}. To show \eqref{enum2}, fix an arbitrary \( R\in \modelX \). Also fix an \( \eps \) such that $R\in \curlySOf{D_{\eps}}$. In particular, $D_{\eps} \subseteq R$ and so $D_{\eps} (t)  \subseteq R(t)$. Now, we have 
	\begin{equation*}
		I D_{\eps} (t) \cap D_{\eps} \subseteq I R(t) \cap D_{\eps} \subseteq I R(t) \cap R
	\end{equation*}
	and extending this inequality to \( R(t) \) yields
	\begin{equation*}
		\left( 	I D_{\eps} (t) \cap D_{\eps}  \right) R(t) \subseteq \left( I R(t) \cap R \right) R(t) .
	\end{equation*}
	On the other hand, by assumption,  
	$\left( 	I D_{\eps} (t) \cap D_{\eps}  \right)  D_{\eps} (t) = ID_{\eps} (t)$ holds and hence we get 
	\begin{equation*}
		IR(t)= \left(  ID_{\eps} (t) \right) R(t) \subseteq \left( I R(t) \cap R \right) R(t)
	\end{equation*}
	which shows $IR(t)$ is relevant.  
	
	Now, assume \eqref{enum2} and fix a ring \( A \in \affinesInX \). Take a polynomial \( f \) from $ IA(t) \cap A[t]$. We need to show that \( IA(t) \) contains coefficients of \( f \). For any \( R\in \curlySOf{A} \),  $f\in R(t)$ and so by  \eqref{enum2}, $IR(t)$ contains the coefficients of \( f \). Therefore, we have
	\begin{equation*}
		\{ r_0 , \dotsc , , r_m \}\subseteq \bigcap_{R \in \curlySOf{A}} IR(t) = IA(t)
	\end{equation*}
	where $f= r_0  + \dotsb + r_m t^m $. Lastly, we note that \eqref{enum1} is a particular case of~\eqref{enum3}.
\end{proof}

\section{Ideals on Projective Models} \label{Sec. Abh vs rel ideals}

Throughout this section, \( \modelX \) will be a fixed model of a field \( K \) over an integral domain \( D \). Our goal is to establish a bijective correspondence between the relevant ideals on \( \dstar \) and the ideals on \( \modelX \) when \( \modelX \) is a projective model. We start by defining a map that sends a relevant ideal to an ideal on \( \modelX \), without requiring projectivity. Afterward, assuming \( \modelX \) is  a projective model, we provide the inverse of the assignment to demonstrate that the map is bijective. 

\begin{definition} 
	For a given relevant ideal \( I \) of \( \dstar \), define the function \( \abyideal \)  on \( \modelX \) by setting 
	\begin{equation*}
		\abyideal (R) = IR(t)\cap R
	\end{equation*}
	for each \( R\in \modelX \). This function \( \abyideal \)  is referred to as the preideal associated with \( I \).
\end{definition}
Note that the function \( \abyideal \)  defined above indeed qualifies as a preideal on \( \modelX \), since $ IR(t)\cap R $ is an ideal of \( R \) for any \( R\in \modelX \). For consistency with our established notation, we will henceforth write \( \abyideal R \)  instead of $\abyideal(R)$.

\begin{lemma} \label{lemma for rel to ideal} 
	Let \( I \) be a relevant ideal of \( \dstar \) and \( \abyideal \)  be the preideal associated with \( I \). Then, for any \( A\in \affinesInX \), the equality $A\cap \abyideal = A \cap IA(t)$ holds.
\end{lemma}
\begin{proof}
	Fix \( A\in \affinesInX \), that is \( A \) is an affine ring such that \( \curlySOf{A} \subseteq \modelX \). Applying Proposition~\ref{p. equiv def for Dstar} to the ideal \( IA(t) \) gives 
	\begin{equation*}
		IA(t) \cap A = \Big(  \bigcap_{R \in \curlySOf{A}} IR(t) \ \Big) \cap A .
	\end{equation*}
	Moreover, using the fact $A= \bigcap_{R \in \curlySOf{A}} R$ 
	we  can write 
	\begin{equation*}
		\Big(  \bigcap_{R \in \curlySOf{A}} IR(t) \ \Big) \cap A = \bigcap_{R \in \curlySOf{A}} \left( IR(t) \cap R \right) = \bigcap_{R \in \curlySOf{A}} \abyideal R 
	\end{equation*}
	which yields the desired equality $IA(t) \cap A = A\cap \abyideal$.
\end{proof}

\begin{prop} \label{rel ideal gives ideal on models}  
	Let \( I \) be a relevant ideal of \( \dstar \) and \( \abyideal \)  be the preideal associated with \( I \). Then, \( \abyideal \)  is an ideal on \( \modelX \). 
\end{prop}
\begin{proof}
	Fix a ring \( A\in \affinesInX \) and \( R\in \curlySOf{A} \). Denote $\fraka$ for the ideal $ IA(t)\cap A$. As \( I \) is relevant, we have the equality $I A(t) = \fraka A(t)$ and by extending it to the ring \( R(t) \), we get $I R(t) = \fraka R(t)$. Now, contracting to \( R \) gives $\abyideal R = \fraka R$ and hence, from Lemma~\ref{lemma for rel to ideal}, we conclude $\abyideal R = \left(A\cap \abyideal \right) R$. 
\end{proof}

Given a projective model \( \modelX \), it is easy to observe that the assignment that sends an ideal $I \ideal \dstar$ to its associated ideal \( \abyideal \)  is injective. Indeed, it is also surjective. To demonstrate this, we now write its inverse. 
\begin{lemma} \label{lemma for from abyideal to rel}  
	Let \( \abyideal \)   be an ideal on \( \modelX \). Consider
	\begin{equation*}
		I \coloneq \bigcap_{R \in \modelX} \left( \abyideal R \right) R(t)
	\end{equation*}
	which is an ideal of \( \dstar \). Then, 
	\begin{equation*}
		I=\bigcap_{\eps \in \Lambda}  \left( D_{\eps} \cap \abyideal \right) D_{\eps} (t)
	\end{equation*}
	for any family $\{D_{\eps}\}_{\eps \in \Lambda}\subseteq \affinesInX$ that defines \( \modelX \).
\end{lemma}
\begin{proof}
	Fix a  family \( \{D_{\eps}\} \subseteq \affinesInX \) that defines \( \modelX \). Denote $\fraka_{\eps}$ for the ideal \( D_{\eps} \cap \abyideal \). As \( \abyideal \)  is an ideal on \( \modelX \), for any \( \eps \) and \( R\in \curlySOf{D_{\eps} } \), we have 
	$ \abyideal R = \fraka_{\eps} R$ and hence, \( \big( \abyideal R \big) R(t) = \fraka_{\eps} R(t) \). Therefore,
	\begin{equation*}
		I  = \bigcap_{\eps \in \Lambda} \bigcap_{R \in \curlySOf{D_{\eps} }} \left( \abyideal R \right) R (t) 
		= \bigcap_{\eps \in \Lambda} \bigcap_{R \in \curlySOf{D_{\eps} }} \fraka_{\eps} R(t) 
		= \bigcap_{\eps \in \Lambda} \fraka_{\eps} D_{\eps} (t)
	\end{equation*}
	where the last equality comes from Proposition~\ref{p. equiv def for Dstar}. 
\end{proof}

\begin{prop} \label{ideal on models gives rel ideal}  
	Assume \( \modelX \) is a projective model. Let \( \abyideal \)   be an ideal on \( \modelX \) and $I = \bigcap_{R \in \modelX} \left( \abyideal R \right) R(t)$. Then \( I \) is relevant. 
\end{prop}
\begin{proof}
	Say \( \modelX \) is determined by $\{x_1,\dotsc, x_n \} \subseteq K\setminus \{0\}$. Take an arbitrary $S\in \modelX$. There exists $\eps \in \{1, \dotsc ,n \}$ and a prime $P\ideal D_{\eps} $ such that the local ring \( S \) is the localization of  \( D_{\eps} \)  at the prime ideal \( P \) where \( D_{\eps} \) denotes \( D[x_1/ x_{\eps} ,\dotsc, x_n / x_{\eps}] \).
	
	Let 
	\begin{equation*}
		y_i=
		\begin{cases}
			x_i + x_{\eps} & \text{if } x_i/x_{\eps} \in P \\
			x_i  & \text{if } x_i/x_{\eps} \notin P
		\end{cases}
	\end{equation*}
	for each \( i=1, \dotsc , n \). Notice that $y_{\eps} = x_{\eps}$ since $x_{\eps}/x_{\eps} =1$ cannot be in the prime ideal \( P \). Therefore, the set $\{y_1, \dotsc , y_n \}$ generates the same \( D \)-module as \( \{ x_1, \dotsc , x_n\} \) and hence we have \( \modelX = \projModel(D; y_1, \dotsc, y_n) \). Let \( A_i \) denote the affine ring $A[y_1/y_i, \dotsc, y_n/y_i]$ for each $i=1, \dotsc ,n $. We claim that \( S \) contains \( A_i \) for all \( i \). We already know that \( S \) contains \( A_{\eps} \)  as $ A_{\eps} = D_{\eps} $. In particular, $y_1/y_{\eps}, \dotsc, y_n/y_{\eps} $ are in \( S \). First, we argue that these elements are invertible in \( S \). Fix some \( j \) from \( \{1 , \dotsc ,  n\} \). If $x_j/x_{\eps} \in P$, then $y_j=x_j+ x_{\eps}$ and so the element
	\begin{equation*}
		\frac{y_j}{y_{\eps}} 
		= \frac{x_j + x_{\eps}}{x_{\eps}}
		= \frac{x_j}{x_{\eps}} + 1
	\end{equation*}
	cannot be in the prime ideal \( P \). In the other case, when $x_j/x_{\eps} \notin P$, we trivially get $y_j/y_{\eps} \notin P$ as $y_j = x_j$. Hence, in both cases the element $y_j/ y_{\eps} \in D_{\eps}  $ is not contained by \( P \) and therefore it is a unit element of \( S \). Thus, all $y_1/y_{\eps}, \dotsc, y_n/y_{\eps} $ and their inverses are in \( S \). Since for all $i ,\, j$ we can write $y_j/ y_{i}$ as a product of these unit elements and their inverses, \( S \) must contain it as well.
	
	Now, since \( \modelX = \projModel(D; y_1, \dotsc, y_n) \), the family \( \{ A_i \}_{i=1}^n \) of affine rings defines \( \modelX \). Therefore, by using Lemma~\ref{lemma for from abyideal to rel} we write 
	\begin{equation*}
		I  =  \bigcap_{i=1}^n \big( A_i \cap \abyideal \big) A_i(t) = \bigcap_{i=1}^n  \fraka_{i} 
	\end{equation*}
	where $\fraka_{i} =  \left( A_i\cap \abyideal \right) A_i(t) \cap \dstar $. 
	On the other hand, \( S(t) \)  is the localization of \( A_{\eps}(t) \)  at the prime ideal \( PA_{\eps}(t) \). Moreover, Proposition~\ref{localization for proj models} tells that \( A_{\eps}(t) \)  is a localization of \( \dstar \). Therefore, \( S(t) \)  is a localization of \( \dstar \) as well. 
	Hence, using the fact that localizing the intersection of finitely many ideals is the same as intersecting their respective localizations, we write
	\begin{equation*}
		IS(t) 
		=  \left( \bigcap_{i=1}^n \fraka_{i} \right) S(t) 
		=  \bigcap_{i=1}^n  \fraka_{i} S(t) .
	\end{equation*}
	Now, for any \( i \), we showed that $A_i \subseteq S$ and hence $\fraka_{i} S(t) = \left( A_i \cap \abyideal \right) S (t)$. By combining this with the assumption that \( \abyideal \)  is an ideal on \( \modelX \), we get
	\begin{equation*}
		IS(t) 
		= \bigcap_{i=1}^n  \left( A_i \cap \abyideal \right) S (t)
		=  \bigcap_{i=1}^n  \left(  \abyideal S \right) S (t)
		=\left(  \abyideal S \right) S (t) .
	\end{equation*}
	Thus, $IS(t)$ is relevant as it is generated by elements from \( S \), namely by the $\abyideal S $.
\end{proof}

By combining Propositions \ref{rel ideal gives ideal on models} and \ref{ideal on models gives rel ideal}, we get the following result.
\begin{thm} \label{t. ab ideals vs rel ideals}
	Let \( \modelX \) be a projective model. Then 
	\begin{align*}
		\{\text{ideals on } \modelX \} 
		&\longrightarrow \{\text{relevant ideals of } \dstar \} 
		\\
		\abyideal 
		&\longmapsto \bigcap_{R \in \modelX} \left( \abyideal R \right) R(t)
	\end{align*}
	gives a bijective correspondence with the inverse $I \mapsto \abyideal$ where $	\abyideal R =  IR(t)\cap R $
	for any $	R \in \modelX$.
\end{thm}
\begin{proof}
	From Propositions \ref{rel ideal gives ideal on models} and \ref{ideal on models gives rel ideal}, both maps are well-defined maps; so we need to show that they are indeed inverses of each other. Fix an ideal \( \abyideal \)  on \( \modelX \). Let \( I \) denote the relevant ideal $ \bigcap_{R \in \modelX} \left( \abyideal R \right) R(t)$ and $\abyideal'$ be the ideal on \( \modelX \) defined as $ \abyideal' R =  IR(t)\cap R $ for any $	R \in \modelX$. We need to show that $\abyideal = \abyideal'$. In the proof of Proposition~\ref{ideal on models gives rel ideal}, it is shown that for any \( R\in \modelX \), the equality $IR(t) = \big( \abyideal R \big) R(t)$ holds. Thus, we conclude $\abyideal = \abyideal' $ as the equation
	\begin{equation*}
		\abyideal R = \big( \abyideal R \big) R(t) \cap R = IR(t) \cap R = \abyideal' R 
	\end{equation*}
	holds for any \( R\in \modelX \).
	
	Now, fix a relevant ideal $I\ideal \dstar$. Let \( \abyideal \)  denote the ideal on \( \modelX \) defined as $ \abyideal R =  IR(t)\cap R $ for any $	R \in \modelX$ and $I'$ the relevant ideal $ \bigcap_{R \in \modelX} \left( \abyideal R \right) R(t)$. To prove $I=I'$, by Corollary~\ref{p. ideal on Dstar is intersection over model} it is enough to argue $IR(t) = I' R(t)$ for any \( R\in \modelX \). 
	Fix \( R\in \modelX \) and an affine ring \( A\in \affinesInX \) such that $R \in \curlySOf{A}$. As before, from Proposition~\ref{ideal on models gives rel ideal}, we have the equality $I'R(t) = \big( \abyideal R \big) R(t)$. On the other hand, by construction, $\abyideal R = IR(t) \cap R$ and therefore, we have $I'R(t) = \big(  IR(t) \cap R \big) R(t) =I R(t)$ where the latter equation comes from the assumption that \( I \) is relevant. Thus, we get the desired equality $I=I'$. 
\end{proof}

\begin{cor} \label{thm. relevant ideals corresponds q-coh sheaf of ideals}
	Let \( (\modelX , \sheafO ) \) be a blow-up scheme of an integral domain \( D \) over a finitely generated ideal. There exists a bijective correspondence between the relevant ideals of the corresponding \( \dstar \) and the quasi-coherent sheaf of ideals on  \( (\modelX , \sheafO )\).
\end{cor} 
\begin{proof}
	Fix a finitely generated ideal \( I \) of \( D \). Recall that a way to construct the blow-up scheme of \( D \) over \( I \) is gluing affine schemes of the rings \( D[I{a_1}\inverse] , \dotsc ,  D[I{a_n}\inverse] \)  where $\{a_1 , \dotsc ,  a_n\}$ is a generating set for the ideal \( I \). For the definition and details, we refer the reader to \cite[414-415]{GortzAlgGeo2020} or to \cite[\S 5.6]{swansonHuneke_IntClosure} for a direct approach without using scheme terminology. Therefore, the scheme \( (\modelX , \sheafO )\) is isomorphic to the scheme corresponding to the projective model over \( D \) defined by \( I \), and so we may assume \( \modelX = \projModel(D;I) \). The rest follows from Proposition~\ref{p. correspondance of ab and qCoh ideals} and Theorem~\ref{t. ab ideals vs rel ideals}.
\end{proof}

\end{document}